\documentclass[12pt]{amsart}

\usepackage{amssymb,amscd, easybmat, MnSymbol}
\usepackage{amsbsy,amsmath,amsfonts,amsthm,extarrows,bm}
\usepackage[english]{babel}
\usepackage{graphicx}
\usepackage[all]{xy}
\usepackage{xfrac}

\usepackage[T1]{fontenc}

\usepackage{xr}
\usepackage[bookmarksopen=false,breaklinks=true,backref=page,pagebackref=true,plainpages=false, hyperindex=true,pdfstartview=FitH,colorlinks=true, pdfpagelabels=true,linkcolor=blue,citecolor=red,urlcolor=red,hypertexnames=false]{hyperref}

\newtheorem{theorem}{Theorem}[section]
\newtheorem{lemma}[theorem]{Lemma}
\newtheorem{corollary}[theorem]{Corollary}
\newtheorem{proposition}[theorem]{Proposition}
\newtheorem{remark}[theorem]{Remark}

\newtheorem{definition}[theorem]{Definition}

\newtheorem{Atheorem}{Theorem}[section]

\newtheorem{Aproposition}[Atheorem]{Proposition}

\newenvironment{sproof}[1]
{\begin{proof}[#1]} {\end{proof}}

\newcommand{\ov}[1]{\overline{#1}}

\newcommand{\Z}{\mathbb Z}

\newcommand{\Q}{\mathbb Q}

\newcommand{\E}{\operatorname{E}}
\newcommand{\GE}{\operatorname{GE}}
\newcommand{\Um}{\operatorname{Um}}
\newcommand{\SL}{\operatorname{SL}}
\newcommand{\K}{\operatorname{K}}
\newcommand{\SK}{\operatorname{SK}}
\newcommand{\sr}{\operatorname{sr}}

\newcommand{\Jac}{\operatorname{Jac}}
\newcommand{\St}{\operatorname{St}}

\newcommand{\ip}{\mathfrak{p}}
\newcommand{\ia}{\mathfrak{a}}
\newcommand{\ord}{\operatorname{ord}}

\makeatletter
\def\paragraph{\@startsection{paragraph}{4}%
  \z@\z@{-\fontdimen2\font}%
  {\normalfont\bfseries}}
\makeatother

\makeatletter
\@namedef{subjclassname@2020}{%
  \textup{2020} Mathematics Subject Classification}
\makeatother

\title{The stable rank of $\mathbb{Z}[x]$ is $3$}
\email{luc.guyot.ge@gmail.com}
\author{Luc Guyot}
\date{\today}

\begin{document}
\pagestyle{plain}
\maketitle
\begin{abstract}
Let $\mathbb{Z}[x]$ be the ring of univariate polynomials over $\Z$ and denote by 
$\sr(\mathbb{Z}[x])$ its stable rank in the sense of Bass.
Grunewald, Mennicke and Vaserstein proved that $$\sr(\mathbb{Z}[x]) = 3.$$ 
As the inequality $\sr(\mathbb{Z}[x]) \le 3$ follows immediately from Bass's stable range theorem, 
the above identity is equivalent to the existence of a non-stable unimodular row of size $3$.
This note addresses minor errors found in the existing proof of the latter fact in the literature.
Using the same methods, we show that 
the unimodular row $(3, x + 1, x^2 + 16)$ is not stable.
\end{abstract}

Rings are assumed to be commutative and unital. Let $R$ be a ring.
A row $(r_1,\dots, r_n) \in R^n$ is \emph{unimodular} if $\sum_i R r_i = R$, 
or equivalently, if the ideal generated by $r_1, \dots, r_n$ is $R$. 
We denote by $\Um_n(R)$ the set of unimodular rows of size $n$. 
A row $(r_1,\dots, r_{n + 1}) \in \Um_{n + 1}(R)$ ($n > 0$) is \emph{stable} if there is 
$(s_1, \dots, s_n) \in R^n$ such that $$(r_1 + s_1 r_{n + 1}, \dots, r_{n} + s _nr_{n + 1})$$ belongs to $\Um_n(R)$. 
An integer $n > 0$ lies in the \emph{stable range} of $R$ if every row in $\Um_{n + 1}(R)$ is stable. 
If $n$ lies in the stable range of $R$, then so does $k$ for every $k > n$ \cite[Lemma 11.3.3]{McCR87}.
The \emph{Bass stable rank $\sr(R)$ of $R$} is the least integer in the stable range of $R$. 
This rank is key in several direct sum cancellation results \cite[Theorems 4.26 and 4.28]{Mag02}; 
stably free $R$-modules of stably free rank at least $\sr(R)$ are free \cite[Corollary 4.23]{Mag02}. 
Bass's stable rank is also used to simplify the computation of the algebraic $K$-groups $K_n(R)$ through 
surjectivity and injectivity theorems. For instance, the value of $\sr(R)$ is an upper bound on the order of 
the matrices than can be used to represent elements of $K_1(R)$ 
\cite[Theorem 10.3]{Mag02}.

Grunewald, Mennicke and Vaserstein proved the following:

\begin{Atheorem}(\cite[Proposition 1.9]{GMV94}) \label{ThGMV94}
$\sr(\Z[x]) = 3$.
\end{Atheorem}

The inequality $\sr(\mathbb{Z}[x]) \le 3$ is provided by Bass's upper bound on the stable rank of finite-dimensional rings. 
Indeed, we have $\sr(\Z[x]) \le \dim_{\text{Krull}}(\Z[x]) + 1$ by Bass's stable range theorem \cite[Corollary 6.7.4]{McCR87} and 
$\dim_{\text{Krull}}(\Z[x]) = 2$  \cite[Proposition 6.5.4]{McCR87}.

The authors of \cite[Proposition 1.9]{GMV94} claimed that the row $(21+ 4x, 12, x^2 + 20)$ is unimodular but not stable, 
which would therefore establish Theorem \ref{ThGMV94}.
However, it is unknown if the previous row is stable: because of a typographical error, one should read 
$(21 + 2x, 12, x^2 + 20)$ instead of $(21 + 4x, 12, x^2 + 20)$. 
Besides, the proof of \cite[Proposition 1.9]{GMV94}
uses, the false statement ``$\SK_1(\mathcal{O}, f\mathcal{O}) \neq 1$'' (see Definition \ref{DefSK1Rel}) where $\mathcal{O} = \Z + \Z \sqrt{-5}$ and $f = 2$; Remark \ref{RemError} below explains this mistake in detail. 
Thus, none of the previous unimodular rows has been proven to be unstable.

In Section \ref{SecProof}, we present a proof of Theorem \ref{ThGMV94} which addresses these shortcomings. 
Our proof follows closely the lines of the original.
It consists in exhibiting a quotient $R$ of $\Z[x]$ such that the special Whitehead group $\SK_1(R)$ (see Definition \ref{DefSK1})
is not trivial.
This suffices to show that $\sr(\Z[x]) > 2$. Indeed, if $\sr(\Z[x]) \le 2$, then the natural map  
$\SK_1(\Z[x]) \rightarrow \SK_1(R)$ would be surjective (Corollary \ref{CorSK1Surjection} below), 
which is impossible as $\SK_1(\Z[x]) = 1$ \cite[Theorem 1]{BHS64}.

In Section \ref{SecUnstable}, we show in addition:

\begin{Aproposition} \label{PropUnstable}
The unimodular row 
$(3, x + 1, x^2 + 16)$ 
of $\Z[x]$ is not stable.
\end{Aproposition}

\paragraph{Acknowledgments.}
I am grateful to Yves Cornulier, Pierre de la Harpe, Bogdan Nica, Pace Nielsen, Alain Valette, Wilberd van der Kallen and the anonymous referee 
for their interesting remarks and suggestions. 
I am particularly indebted to Pace Nielsen whose suggestions significantly influenced this note's presentation.
Among other improvements, Pace Nielsen proposed the simpler unimodular row of Proposition \ref{PropUnstable} (compare with Proposition \ref{PropUnstableIntermediate}).

This note originates from stimulating discussions with T. Y. Lam regarding the results of \cite{GMV94} 
and some comments I wrote in \cite{MO18}. His encouragements and his kind support have been extremely well appreciated. 

\section{Bass's stable rank, and surjective homomorphisms} \label{SecSr}
In this section, we prove Corollary \ref{CorSK1Surjection}, which is a simple, but important argument in the proof of
Theorem \ref{ThGMV94}. 
We also introduce the relative special Whitehead group $\SK_1(R, I)$ with Definition \ref{DefSK1Rel}, the group $\K_2(R)$ and an exact sequence, 
called the \emph{relative sequence}, which binds these groups.
This sequence is used to show the existence of an isomorphism $\SK_1(R) \simeq \SK_1(R, I)$ for suitable $R$ and $I$, 
in the proof of Theorem \ref{ThGMV94}, see Lemma \ref{LemSK1Rel} below.

\subsection*{Bass's stable rank}
Recall that rings are supposed to be commutative and unital.
The Bass stable rank can be characterized in terms of a lifting property for unimodular rows of quotient rings. 

\begin{proposition} \label{PropStableRankAndSurjectivity}
Let $R$ be a ring. Then the following are equivalent.
\begin{itemize}
\item[$(i)$] $\sr(R) \le n$.
\item[$(ii)$] For every ideal $I \subseteq R$, the map 
$\Um_n(R) \rightarrow \Um_n(R/I)$ sending $(r_1, \dots, r_n)$ to $(r_1 + I, \dots, r_n + I)$ is surjective. 
\end{itemize}
\end{proposition}

\begin{proof}
$(i) \Rightarrow (ii).$ Let $I$ be an ideal of $R$ and let $(r_1, \dots, r_n) \in R^n$ be such that 
the ideal generated by $r_1 + I, \dots, r_n + I$ is $R/I$. 
By hypothesis, there are $(s_1, \dots, s_n) \in R^n$ and $a \in I$ such that 
$s_1r_1 + \dots + s_n r_n + a = 1$. Thus $(r_1, \dots, r_n, a) \in \Um_{n + 1}(R)$. Since $\sr(R) \le n$, we can find 
$(\lambda_1, \dots, \lambda_n) \in R^n$ such that $(r_1 + \lambda_1 a, \dots, r_n + \lambda_n a) \in \Um_n(R)$, 
which yields the result.

$(ii) \Rightarrow (i).$ Let $\mathbf{r} = (r_1, \dots, r_{n + 1}) \in \Um_{n + 1}(R)$. 
We certainly have $(r_1 + I, \dots, r_n + I) \in \Um_n(R/I)$ for 
$I = Rr_{n + 1}$. Assumption $(ii)$ entails that there is $(\lambda_1, \dots, \lambda_n) \in R^n$ such that 
$(r_1 + \lambda_1 r_{n + 1}, \dots, r_n + \lambda_n r_{n + 1}) \in \Um_n(R)$, which shows that $\mathbf{r}$ is stable.
\end{proof}

Specializing $n$ to $1$, we obtain:

\begin{proposition}\cite[Lemma 6.1]{EO67} \label{SrOneSurjectivity}
Let $R$ be a ring. Then the following are equivalent.
\begin{itemize}
\item[$(i)$] $\sr(R) \le 1$.
\item[$(ii)$] For every ideal $I \subseteq R$, the natural map $R^{\times} \rightarrow (R/I)^{\times}$ is surjective. 
\end{itemize}
\end{proposition}

The condition $\sr(R) \le 1$ received special attention in \cite{EO67} and \cite{Vas84}.
We shall see with Proposition \ref{PropSr2} below that the condition $\sr(R) \le 2$ can also be interpreted in terms 
of surjective group homomorphisms.
We denote by $\Jac(R)$, the \emph{Jacobson radical of $R$}, that is, the intersection of the maximal ideals of $R$.
We record the following proposition for later use.

\begin{proposition}{(\cite[Lemma 11.4.6]{McCR87} and \cite[Exercise 4D.8]{Mag02})}  \label{PropRank}
Let $I$ be an ideal of $R$. Then we have
$\sr(R/I) \le \sr(R)$ and equality holds if $I \subseteq \Jac(R)$.
\end{proposition}

\subsection*{Rings of stable rank at most $2$}
Rings of stable rank at most $2$ enjoy the following characterization.
\begin{proposition} \label{PropSr2}
Let $R$ be a ring. Then then following are equivalent:
\begin{itemize}
\item[$(i)$] $\sr(R) \le 2$.
\item[$(ii)$] The natural map $\SL_2(R) \rightarrow \SL_2(\ov{R})$ is surjective for every quotient 
$\ov{R}$ of $R$.
\item[$(iii)$] The natural map $\SL_n(R) \rightarrow \SL_n(\ov{R})$ is surjective for every quotient 
$\ov{R}$ of $R$ and every $n \ge 2$.
\end{itemize}
\end{proposition}

Our proof of Proposition \ref{PropSr2} relies on

\begin{lemma}(\cite[Lemma 6.2]{GMV94}, \cite[Corollary 8.3]{EO67}) \label{LemStableRow}
Let $R$ be a ring. Let $(a, b, c) \in \Um_3(R)$ and let $r \mapsto \ov{r}$ denote the natural map from $R$ onto $R/Rc$.
Then the following are equivalent:
\begin{itemize}
\item[$(i)$] The row $(a, b, c)$ is stable.
\item[$(ii)$] Every matrix  
$\begin{pmatrix} \ov{a} & \ov{b} \\  \ov{d} & \ov{e} \end{pmatrix} \in\SL_2(R/Rc)$ has a lift in $\SL_2(R)$.
\end{itemize}
\end{lemma}

\begin{proof}
$(i) \Rightarrow (ii)$. Let $A = \begin{pmatrix} \ov{a} & \ov{b} \\ \ov{d} &  \ov{e} \end{pmatrix} \in \SL_2(R/Rc)$.
Since $(a, b, c)$ is stable, we can find $(r, s) \in \Um_2(R)$ such that $\ov{r} = \ov{a}$ and $\ov{s} = \ov{b}$.   
Let $u, v \in R$ be such that $\begin{pmatrix} r & s \\ u & v \end{pmatrix} \in \SL_2(R)$. 
Then 
$\begin{pmatrix}\overline{a} & \overline{b}\\ \overline{u} & \overline{v}\end{pmatrix}
A^{-1} = \begin{pmatrix} 1 & 0 \\ \ov{w} & 1 \end{pmatrix}$ for some $w \in R$. Therefore 
$\begin{pmatrix} 1 & 0 \\ w & 1 \end{pmatrix}^{-1}
\begin{pmatrix} r & s \\ u & v \end{pmatrix}$ is a lift of $A$.

$(ii) \Rightarrow (i)$. By assumption, we can find $(r, s) \in \Um_2(R)$ such that 
$\overline{r} = \overline{a}$ and $\overline{s} = \overline{b}$, i.e.,  we have 
$r = a + \lambda c$ and $s = b + \mu c$ for some $\lambda, \mu \in R$. 
It follows immediately that $(a, b, c)$ is stable. 
\end{proof}
 
Let $R$ be a ring. Let $n \ge 2$ and denote by $I_n$ the $n \times n$ identify matrix. 
For $1 \le i, j \le n$, let $\epsilon_{ij}$ be the $n \times n$ matrix whose $(i, j)$ entry is $1$ and whose other entries are zero.
For $a \in R$ and $i \neq j$, we set $e_{ij}(a) = I_n + a \epsilon_{ij}$ and call any such matrix an \emph{elementary matrix}.
Let us denote by $\E_n(R)$ the subgroup of $\SL_n(R)$ generated by the elementary matrices.
 
\begin{sproof}{Proof of Proposition \ref{PropSr2}}
Clearly, we have $(iii) \Rightarrow (ii)$. The implication $(ii) \Rightarrow (i)$ is given by Lemma \ref{LemStableRow}.
Hence it only remains to show that $(i) \Rightarrow (iii)$.

Assume that $(i)$ holds and let $n \ge 2$, $A \in \SL_n(\ov{R})$ where $\ov{R} = R/I$ for some ideal $I$ of $R$. 
If $n = 2$, then it follows from Lemma \ref{LemStableRow} that $A$ has a lift in $\SL_2(R)$.

Indeed, write $A = \begin{pmatrix} a + I & b + I \\ d + I &  e + I \end{pmatrix} \in \SL_2(\ov{R})$. 
Putting $c:=1-(ae-bd)\in I$, then $(a,b,c)\in \Um_3(R)$ and 
$\widetilde{A} = \begin{pmatrix} a + Rc & b + Rc \\ d + Rc &  e + Rc \end{pmatrix} \in \SL_2(R/Rc)$.
By hypothesis, we have $\sr(R) \le 2$, so that Lemma \ref{LemStableRow} applies and
provides us with a lift of $\widetilde{A}$ in $\SL_2(R)$ which is also a lift of $A$.

We can now assume that $n > 2$ and proceed by induction on $n$.
Since $\sr(\ov{R}) \le 2$ by Proposition \ref{PropRank}, and because the first row of $A$ is unimodular, 
we can find $E, E' \in \E_n(\ov{R})$ 
such that $A = E  \begin{pmatrix} 1 & 0 \\ 0 &A' \end{pmatrix} E'$ with $A' \in \SL_{n - 1}(\ov{R})$. 
By induction hypothesis, the matrix $A'$ has a lift in $\SL_{n - 1}(R)$.
Clearly, the matrices $E$ and $E'$ lift to $\SL_n(R)$. Thus $A$ has a lift in $\SL_n(R)$.
\end{sproof}

The next result refers to the \emph{special Whitehead group} $\SK_1(R)$ of a ring $R$. 
We define this group as follows.
Let $\E(R) := \bigcup_n \E_n(R)$ and $\SL(R) := \bigcup_n \SL_n(R)$ be the ascending unions for which 
the embeddings $\E_n(R) \rightarrow \E_{n + 1}(R)$ and  $\SL_n(R) \rightarrow \SL_{n + 1}(R)$ are defined 
through $A \mapsto \begin{pmatrix} A & 0 \\ 0 & 1 \end{pmatrix}$. 
Then $\E(R)$  is a normal subgroup of $\SL(R)$ \cite[Whitehead Lemma 9.7]{Mag02}.

\begin{definition} \label{DefSK1}
$\SK_1(R) := \SL(R)/\E(R)$.
\end{definition}

The following corollary is an immediate consequence of Proposition \ref{PropSr2}.
\begin{corollary} \label{CorSK1Surjection}
Let $R$ be a ring satisfying $\sr(R) \le 2$. Let $\ov{R}$ be a quotient of $R$.
Then the natural map $\SK_1(R) \rightarrow \SK_1(\ov{R})$ is surjective.
\end{corollary}

Following \cite{Coh66}, we call a ring $R$ a \emph{$\GE_2$-ring} if $\SL_2(R) = \E_2(R)$.

The next corollaries are of independent interest.
\begin{corollary} \label{CorGE}
Let $R$ be a $\GE_2$-ring satisfying $\sr(R) \le 2$.
Then every quotient of $R$ is a $\GE_2$-ring.
\end{corollary}

\begin{proof}
Let $\overline{R}$ be a quotient of $R$ and let $A \in \SL_2(\overline{R})$. 
By Proposition \ref{PropSr2}, the matrix $A$ is the image of some matrix in $\SL_2(R)$.
Since $\SL_2(R) = \E_2(R)$ by assumption, we infer that $A \in \E_2(\overline{R})$.
\end{proof}

\begin{remark}
Corollary \ref{CorGE} provides a straightforward proof of \cite[Theorem A]{Guy18}. 
\end{remark}

\begin{corollary} \cite[Theorem 3.6]{McGov08}
Let $R$ be a ring such that any proper quotient of $R/\Jac(R)$ has stable rank $1$. Then $\sr(R) \le 2$.
\end{corollary}
\begin{proof}
It is easy to check that rings of stable rank $1$ are $\GE_2$-rings. The result follows by
combining Propositions \ref{PropSr2} and \ref{PropRank}.
\end{proof}

\subsection*{The groups $\SK_1(R, I)$, $\K_2(R)$ and the relative sequence}
In this section, we define the relative special Whitehead group $\SK_1(R, I)$ 
for $I$ an ideal of a ring $R$ and the group $\K_2(R)$ from algebraic K-theory. 
We describe then the \emph{relative sequence}, an exact sequence relating them to one another.
This sequence will come in handy for the K-theoretical computations of Section \ref{SecProof}.

Recall that  $\E_n(R)$  is the subgroup of $\SL_n(R)$ generated by the elementary matrices 
$e_{ij}(r)$ with $r \in R$. 
Let $I$ be an ideal of $R$. We denote by $\E_n(R, I)$ the normal subgroup of $\E_n(R)$ 
which is normally generated by the matrices $e_{ij}(a)$ with $a \in I$, $1 \le i \neq j \le n$.
We denote by $\SL_n(R, I)$ the kernel of the natural map $\SL_n(R) \rightarrow \SL_n(R/I)$.
Let $\E(R, I) := \bigcup_n \E_n(R, I)$ and $\SL(R, I) := \bigcup_n \SL_n(R, I)$ be the ascending unions for which 
the embeddings $\E_n(R, I) \rightarrow \E_{n + 1}(R, I)$ and  $\SL_n(R, I) \rightarrow \SL_{n + 1}(R, I)$ are defined 
through $A \mapsto \begin{pmatrix} A & 0 \\ 0 & 1 \end{pmatrix}$. 
Then $\E(R, I)$  is a normal subgroup of $\SL(R, I)$ \cite[Relative Whitehead Lemma 11.1]{Mag02}.

\begin{definition} \label{DefSK1Rel}
$\SK_1(R, I) := \SL(R, I)/\E(R, I)$.
\end{definition}

We now turn to the definition of $\K_2(R)$. The \emph{Steinberg group} $\St_n(R)$ 
is the group with generators $x_{ij}(r)$, with $i \neq j$, $1\le i, j \le n$ and $r \in R$ subject to the defining relations:
\begin{equation}
\begin{split}
x_{ij}(r) x_{ij}(s)  = x_{ij}(r + s), \\ 
\lbrack x_{ij}(r), x_{kl}(s)\rbrack = 
\left\{ \begin{array}{c} 1, \text {if } i \neq l, j \neq k, \\ x_{il}(rs), \text{ if } i \neq l, j = k.\end{array}\right.
\end{split}
\end{equation}

Removing the restrictions $i \le n$ and $j \le n$ on the generators $x_{ij}(r)$, 
the same presentation defines the \emph{Steinberg group} $\St(R)$. 
Because the elementary matrices $e_{ij}(r)$ obey the above standard relations and generate $\E(R)$, 
the map $x_{ij}(r) \mapsto e_{ij}(r)$ induces a surjective group homomorphism 
$\St(R) \twoheadrightarrow \E(R)$. The kernel of this homomorphism is $\K_2(R)$. 

Finally, we introduce the so-called \emph{relative sequence}:
\begin{equation} \label{EqExact}
\K_2(R/I) \xrightarrow{\partial_1} \SK_1(R, I) \rightarrow \SK_1(R) \rightarrow \SK_1(R/I)
\end{equation}
where the second and third arrows are induced respectively by the inclusion 
$\SL(R, I) \subseteq \SL(R)$ and the natural map $\SL(R) \rightarrow \SL(R/I)$. 
For the connecting homomorphism $\partial_1$, 
we refer the reader to \cite[Theorem 13.20 and Example 13.22]{Mag02} where this sequence is shown to be exact.

\section{Proof of Theorem \ref{ThGMV94}} \label{SecProof}

We shall establish

\begin{proposition} \label{PropSK1neq1}
$\SK_1(\Z + 4i \Z) \neq 1$ where $i = \sqrt{-1}$.
\end{proposition}
Observe that $\Z + 4i \Z$ naturally identifies with $\Z[x]/(x^2 + 16)$.
As outlined in the introduction, Theorem \ref{ThGMV94} immediately follows from the combination of
Proposition \ref{PropSK1neq1} with Corollary \ref{CorSK1Surjection} and the following theorem:
\begin{theorem}{\cite[Theorem 1]{BHS64}} \label{ThSK1Zx}
$\SK_1(\Z[x]) = 1$.
\end{theorem}

Our proof of Proposition \ref{PropSK1neq1} relies on

\begin{lemma} \label{LemSK1Rel}
Let $I$ be an ideal of a ring $R$ such that $R/I \simeq \Z /n \Z$ for some $n \ge 0$. 
Then the natural map 
\begin{equation} \label{EqRItoR}
\SK_1(R, I) \rightarrow \SK_1(R)
\end{equation}
induced by the inclusion $\SL(R, I) \subseteq \SL(R)$ 
is an isomorphism.
\end{lemma}

\begin{proof} 
In the exact sequence (\ref{EqExact}) introduced in Section \ref{SecSr}
$$\K_2(R/I) \xrightarrow{\partial_1} \SK_1(R, I) \rightarrow \SK_1(R) \rightarrow \SK_1(R/I)$$
the last term, namely $\SK_1(R/I)$, is trivial since $R/I$ is Artinian and hence of stable rank $1$ \cite[Corollary 10.5]{Bas64}. 
In addition, the image of $\K_2(R/I)$ in $\SK_1(R, I)$ is also trivial. 
Indeed $\K_2(R/I)$ is generated by the Steinberg symbol $\{-1 + I, -1 + I\}_{R/I}$ 
because $R/I \simeq \mathbf{Z}/n \mathbf{Z}$ \cite[Definition 12.23, Exercises 12B.6 and 13A.9]{Mag02}. 
As $\{-1, -1\}_R$ is a lift of the previous symbol in $\K_2(R)$ \cite[Exercise 12B.7]{Mag02} our claim follows from 
the definition of $\partial_1$ \cite[Theorem 13.20]{Mag02}, which completes the proof. 
\end{proof}

\begin{sproof}{Proof of Proposition \ref{PropSK1neq1}}
Let $S = \mathbb{Z}[i], I = 4S$ and $R = \Z + I$.
By Lemma \ref{LemSK1Rel} we have $\SK_1(R) \simeq \SK_1(R, I)$. 
The inclusion $R \subset S$ induces the identity map on $\SL(R, I) = \SL(S, I)$ 
and hence a surjective group homomorphism 
\begin{equation} \label{EqRItoOI}
\SK_1(R, I) \twoheadrightarrow \SK_1(S, I).
\end{equation}

As $\SK_1(S, I) \simeq \Z /2\Z$ by the Bass-Milnor-Serre Theorem \cite[Theorem 11.33]{Mag02} 
(Theorem \ref{ThBMS} below), we conclude that $\SK_1(R) \neq 1$.
\end{sproof}

\begin{remark}
As an alternative to Proposition \ref{PropSK1neq1}, one can show that 
$$\SK_1(\mathbb{Z} + 3\zeta_3 \mathbb{Z}) \neq 1, \text { where } \zeta_3 = e^{\frac{2i\pi}{3}}.$$ 
This is a direct consequence of  \cite[Lemma 3.2 and subsequent remark]{Swa71} and Lemma \ref{LemSK1Rel}.
\end{remark}

\section{Proof of Proposition \ref{PropUnstable}} \label{SecUnstable}

We shall prove that the unimodular row
$$(3, x + 1, x^2 + 16) \in \Um_3(\mathbb{Z}[x])$$ is not stable.
We apply the strategy devised in \cite[Proof of Proposition 1.9]{GMV94}: we look for an explicit matrix in $\SL_2(R)$ that defines a 
non-trivial element of $\SK_1(R)$ for a suitable quotient $R$ of $\Z[x]$. To do so, we resort to  
the Bass-Milnor-Serre Theorem \cite[Theorems 3.6 and 4.1]{BMS67} and its description of $\SK_1(S)$ for $S$ 
the ring of integers of a totally imaginary number field, in terms of power residue symbols.
The following definitions are required to state the latter theorem.

Let $m > 0$ be a rational integer. We denote by $\mu_m$ the group of $m$-th roots of unity in the field of complex numbers.
Let $S$ be the ring of integers of a number field and suppose that $S$ contains $\mu_m$ for some $m > 0$.
For $b \in S$ and $\ia$ an ideal of $S$ such that $\ia + S b m = S$,  define the \emph{$m$-th power residue symbol}
$$
\left(\frac{b}{\ia}\right)_m := \prod_{\ip \vert \ia} \left(\frac{b}{\ip}\right)_m^{\ord_{\ip}(\ia)}
$$
where $\ip$ ranges in the set of prime ideals of $S$ dividing $\ia$ and where 
$ \left(\frac{b}{\ip}\right)_m$ is the unique element of $\mu_m$ satisfying 
$$ b^{\frac{q - 1}{m}} \equiv \left(\frac{b}{\ip}\right)_m  \mod \ip$$ 
where $q$ is the number of elements in the residue field of $\ip$ 
\cite[Appendix on number theory, pages 86 and 89]{BMS67} \cite[Theorem 11.33 and Proposition 15.40]{Mag02}. 
Note that $m$ divides $q - 1$ by virtue of Lagrange's theorem.
If $\ia = S a$ with $a \in S$, we simply write $\left(\frac{b}{a}\right)_m$.
The latter symbol depends on $b$ only modulo $a$ and is evidently multiplicative in $b$, i.e., we have 
\begin{equation} \label{EqPowerMult}
\left(\frac{bc}{a}\right)_m = \left(\frac{b}{a}\right)_m  \left(\frac{c}{a} \right)_m
\end{equation}
for every $b, c \in S$ such that $Sa + Sbcm = S$.

The following computation will come soon in handy.
\begin{lemma} \label{LemPowerRes}
Let $\ip$ be the principal ideal of $\Z[i]$ generated by $1 + 4i$. Then $\ip$ is prime, $12 \notin \ip$ and we have  
$\left(\frac{12}{1 + 4i}\right)_2 = -1$. 
\end{lemma}

\begin{proof}
As the norm $\operatorname{N}_{\Q(i)/\Q}(1 + 4i)$ is $17$, 
the ideal $\ip$ is prime and 
$\Z[i] /\ip$ is the field with $17$ elements. Clearly, we have $12 \notin \ip$ so 
that $\left(\frac{12}{1 + 4i}\right)_2$ is well-defined.
It follows immediately from the multiplicative law (\ref{EqPowerMult}) that 
$\left(\frac{12}{1 + 4i}\right)_2 = \left(\frac{3}{1 + 4i}\right)_2$.
  
Observing that $3^{\frac{17 - 1}{2}} = 3^8 \equiv -1 \mod 17$, we infer that $\left(\frac{3}{1 + 4i}\right)_2 = -1$, 
which completes the proof.
\end{proof}

\begin{theorem} \cite[Theorems 3.6 and 4.1]{BMS67} \label{ThBMS}
Let $S$ be the ring integers of a totally imaginary number field. 
Let $m = m(S)$ denote the number of roots of unity in $S$. If $I$ is a non-zero ideal of $S$, define the divisor $r = r(I)$ of $m$ by 
$$\ord_p(r) = j_p(I),  \text{ for every prime divisor } p \text{ of } m,$$ where $ j_p(I)$ is the nearest integer in the interval 
$\lbrack 0, \ord_p(m) \rbrack$ to 
$$\min_{\ip \vert p} \left\lfloor \frac{\ord_{\mathfrak{p}}(I)}{\ord_{\ip}(S p)} - \frac{1}{p - 1} \right\rfloor$$ 
(where $\lfloor x \rfloor$ denotes the greatest integer $\le x$ and $\ip$ ranges over the prime ideals of $S$ containing $p$).

Then the map $\begin{pmatrix} a & b \\ * & *\end{pmatrix} \in \SL_2(S, I) \mapsto (\frac{b}{a})_r$ induces an isomorphism from $\SK_1(S, I)$ onto $\mu_r$.
\end{theorem}

The first item of the next proposition was implicit in the above theorem. 

\begin{proposition} \cite[Lemma 11.24 and Proposition 11.25]{Mag02} \label{PropMennickSymbol}
Let $R$ be a ring and let $I$ be an ideal of $R$.
\begin{itemize} 

\item[$(i)$] If two matrices of $\SL_2(R, I)$ 
have the same first row, then they represent the same element of $\SK_1(R, I)$.
For $A = \begin{pmatrix}
a & b \\
* & *
\end{pmatrix} \in \SL_2(R, I)$, define
$$[a, b]_I = A \cdot \E(R, I) \in \SK_1(R, I).$$

\item[$(ii)$]Let $(a, b) \in \Um_2(R)$ with $a \in 1 + I$ and $b \in I$. 
Then $(a, b)$ is the first row of some matrix in $\SL_2(R, I)$.

\item[$(iii)$] If $(a, b')$ is the first row of a matrix in $\SL_2(R, I)$, then we have 
$$[a, bb']_I = [a, b]_I[a, b']_I.$$
\end{itemize}
\end{proposition}

\begin{remark} \label{RemError}
Set $S = \mathbb{Z}[\sqrt{-5}]$ and $I = 2S$. 
The statement \cite["$f = 2$ has property ($\ast$)" on page 191]{GMV94} implies that 
$\SK_1(S, I) \simeq \mu_2$, which is false. 
Indeed, we have $m(S) = 2$ and it follows from Theorem \ref{ThBMS} 
that $r(I) = 1$ and hence $\SK_1(S, I) = 1$. 
\end{remark}

As a stepping stone to Proposition \ref{PropUnstable}, we shall establish:

\begin{proposition} \label{PropUnstableIntermediate}
The unimodular row $(12, x + 1, x^2 + 16)$ 
of $\Z[x]$ is not stable.
\end{proposition}

For the remainder of this section, we set 
$$S = \mathbb{Z}[i], I =  4S, \text{ and } R = \Z + I,$$
with a view to apply Theorem \ref{ThBMS} and Proposition \ref{PropMennickSymbol}.
The reason why we first consider the row $(12, x + 1, x^2 + 16)$, and not $(3, x + 1, x^2 + 16)$,
is that $[1 + 4i, 12]_I$ is well-defined since $12 \in I$ whereas $3 \notin I$.

\begin{sproof}{Proof of Proposition \ref{PropUnstableIntermediate}}
By Theorem \ref{ThBMS}, we have $r(I) = 2$. 
By Theorem \ref{ThBMS} and Proposition \ref{PropMennickSymbol}.$ii$, the power residue symbol 
$\binom{12}{1 + 4i}_2$ is the image of $A \cdot \operatorname{E}(S, I) \in \SK_1(S, I)$ for some matrix  
$A = \begin{pmatrix} 1 + 4i & 12 \\ * & *\end{pmatrix} \in \SL_2(S, I)$. 
By Lemma \ref{LemPowerRes}, we have $A \cdot \operatorname{E}(S, I) \neq 1$.
The matrix $A$ can be lifted from $\SK_1(S, I)$ to  $\SK_1(R, I)$ via (\ref{EqRItoOI}), 
and certainly maps to a non-trivial element of $\SK_1(R)$ via the isomorphism (\ref{EqRItoR}). 

Considering the surjective ring homomorphism from $\Z[x]$ onto $R$ induced by $x \mapsto 4i$, 
we infer from Theorem \ref{ThSK1Zx} and 
Lemma \ref{LemStableRow} that the row $(1 + x, 12, x^2 + 16)$ cannot be stable. This trivially implies the result. 
\end{sproof}

\begin{sproof}{Proof of Proposition \ref{PropUnstable}}
Let $\gamma = [1 + 4i, 12]_{I} [1 + 4i, 4]_{I}^{-1} \in \SK_1(R, I)$. 
Thanks to Proposition \ref{PropMennickSymbol}.$iii$, we observe that 
$\gamma$ is the image of $[1 + 4i, 3]_{R} \in \SK_1(R, R) = \SK_1(R)$ by the inverse of the isomorphism (\ref{EqRItoR}). 
Mapping $\gamma$ to $\SK_1(S, I)$ via (\ref{EqRItoOI}), we obtain again, by means of  Theorem \ref{ThBMS}, 
a non-trivial element of $\SK_1(S, I)$ since $\left(\frac{4}{1 + 4i}\right)_2 = 1$.
Reasoning as in the proof of  Proposition \ref{PropUnstableIntermediate}, we conclude that $(3, x + 1, x^2+16)$ is not stable.
\end{sproof}


\bibliographystyle{alpha}
\bibliography{Biblio}

\def\cprime{$'$} \def\cprime{$'$} \def\cprime{$'$} \def\cprime{$'$}
  \def\cprime{$'$}
\begin{thebibliography}{GMV94}

\bibitem[Bas64]{Bas64}
H.~Bass.
\newblock {$K$}-theory and stable algebra.
\newblock {\em Inst. Hautes \'Etudes Sci. Publ. Math.}, (22):5--60, 1964.

\bibitem[BHS64]{BHS64}
H.~Bass, A.~Heller, and R.~G. Swan.
\newblock The {W}hitehead group of a polynomial extension.
\newblock {\em Inst. Hautes \'{E}tudes Sci. Publ. Math.}, (22):61--79, 1964.

\bibitem[BMS67]{BMS67}
H.~Bass, J.~Milnor, and J.-P. Serre.
\newblock Solution of the congruence subgroup problem for {${\rm
  SL}_{n}\,(n\geq 3)$} and {${\rm Sp}_{2n}\,(n\geq 2)$}.
\newblock {\em Inst. Hautes \'{E}tudes Sci. Publ. Math.}, (33):59--137, 1967.

\bibitem[Coh66]{Coh66}
P.~M. Cohn.
\newblock On the structure of the {${\rm GL}_{2}$} of a ring.
\newblock {\em Inst. Hautes \'Etudes Sci. Publ. Math.}, (30):5--53, 1966.

\bibitem[EO67]{EO67}
D.~Estes and J.~Ohm.
\newblock Stable range in commutative rings.
\newblock {\em J. Algebra}, 7:343--362, 1967.

\bibitem[GMV94]{GMV94}
F.~Grunewald, J.~Mennicke, and L.~N. Vaserstein.
\newblock On the groups {${\rm SL}_2({\bf Z}[x])$} and {${\rm SL}_2(k[x,y])$}.
\newblock {\em Israel J. Math.}, 86(1-3):157--193, 1994.

\bibitem[Guy18a]{MO18}
L.~Guyot.
\newblock Bass stable range of $\mathbf{Z}[x]$.
\newblock MathOverflow, February 2018.
\newblock URL:https://mathoverflow.net/q/250088 (version: 2018-02-18).

\bibitem[Guy18b]{Guy18}
L.~Guyot.
\newblock On quotients of generalized {E}uclidean group rings.
\newblock {\em Comm. Algebra}, 46(3):1116--1120, 2018.

\bibitem[Mag02]{Mag02}
B.~A. Magurn.
\newblock {\em An algebraic introduction to {$K$}-theory}, volume~87 of {\em
  Encyclopedia of Mathematics and its Applications}.
\newblock Cambridge University Press, Cambridge, 2002.

\bibitem[McG08]{McGov08}
W.~Wm. McGovern.
\newblock B\'{e}zout rings with almost stable range 1.
\newblock {\em J. Pure Appl. Algebra}, 212(2):340--348, 2008.

\bibitem[MR87]{McCR87}
J.~C. McConnell and J.~C. Robson.
\newblock {\em Noncommutative {N}oetherian rings}.
\newblock Pure and Applied Mathematics (New York). John Wiley \& Sons, Ltd.,
  Chichester, 1987.
\newblock With the cooperation of L. W. Small, A Wiley-Interscience
  Publication.

\bibitem[Swa71]{Swa71}
R.~G. Swan.
\newblock Excision in algebraic {$K$}-theory.
\newblock {\em J. Pure Appl. Algebra}, 1(3):221--252, 1971.

\bibitem[Vas84]{Vas84}
L.~N. Vaserstein.
\newblock Bass's first stable range condition.
\newblock In {\em Proceedings of the {L}uminy conference on algebraic
  {$K$}-theory ({L}uminy, 1983)}, volume~34, pages 319--330, 1984.

\end{thebibliography}

\end{document}